\documentclass{article}
\usepackage[utf8]{inputenc}

\usepackage{graphicx, amssymb, amsfonts, latexsym, amsthm, amsmath, framed,float}
\usepackage{hyperref}
\usepackage{upgreek, tipa, url, amsthm, subfigure, enumitem, setspace, mathrsfs, verbatim, color}
\usepackage{overpic}
\usepackage{pdfpages}
\usepackage{cancel}

\theoremstyle{plain}
\newtheorem{thm}{Theorem}[section]
\newtheorem{cor}[thm]{Corollary}
\newtheorem{lem}[thm]{Lemma}
\newtheorem{claim}[thm]{Claim}
\newtheorem{prop}[thm]{Proposition}

\theoremstyle{definition}
\newtheorem{defn}[thm]{Definition}

\theoremstyle{remark}

\newcommand{\ord}{\mathcal{ON}}
\newcommand{\lex}{\text{lex}}

\newcommand{\restr}{\upharpoonright}

\begin{document}
	
	\title{Algorithmic traversals of infinite graphs}
	\author{Siddharth Bhaskar\footnote{Department of Computer Science, Haverford College, Haverford PA} and Anton Jay Kienzle\footnote{Undergraduate, Haverford College, Haverford, PA}}
    \date{}
	\maketitle
    
    \begin{abstract}
    A \emph{traversal} of a connected graph is a linear ordering of its vertices all of whose initial segments induce connected subgraphs. Traversals, and their refinements such as breadth-first and depth-first traversals, are computed by various \emph{graph searching} algorithms. We extend the theory of generic search \cite{CK08} and breadth-first search from finite graphs to wellordered infinite graphs, recovering the notion of ``search trees'' in this context. We also prove tight upper bounds on the extent to which graph search and breadth-first search can modify the order type of the original graph, as well as characterize the traversals computed by these algorithms as lexicographically minimal. 
    \end{abstract}
	
    \paragraph{Acknowledgements} We would like to thank Steven Lindell and Scott Weinstein for teaching us the correct notion of traversal, and the importance of distinguishing searches (algorithms) from traversals (what they compute). We also thank Sherwood Hachtman for many illuminating conversations and for being a reliable oracle for ZFC.
    
    Finally, we are indebted to Haverford College and the Zweifler Family Fund for Student Researches in the Sciences for supporting the research project of Jay Kienzle with Siddharth Bhaskar during the summer of 2018. This paper was born from that collaboration.
    
	\section{Graph search and traversals} \label{intro}
    The idea of ``searching through'' a graph pervades the theory of graph algorithms. To quote Corneil and Krueger \cite{CK08},
    
    ``Graph searching is fundamental. Most graph algorithms employ some mechanism for systematically visiting all vertices and edges in the given graph. After choosing an initial vertex, a search of a connected graph visits each of the vertices and edges of the graph such that a new vertex is visited only if it is adjacent to some previously visited vertex.''
    
    In the present investigation, we ask to what extent the notion of graph searching can be lifted to  general infinite graphs. In doing so, it is particularly important to distinguish between the search \emph{algorithm} and the resulting \emph{vertex orderings} that it computes. We call the various algorithms \emph{searches} and the vertex orderings they compute \emph{traversals}.
    
    The most basic definition in our investigation is that of a plain traversal, and the present formulation is due to Diestel \cite{D16}, as well as Lindell and Weinstein.
    
   \begin{defn}
   Given a connected graph $(V,E)$, a \emph{traversal} is a linear ordering $<$ of $V$ such that the subgraph induced by any initial segment\footnote{i.e., subset of $V$ which is downward closed under $<$} is connected. A traversal of an arbitrary graph is a linear ordering of its vertices in which the connected components form intervals, and each such interval is a traversal of its corresponding component.
   \end{defn}
   
   Notice that this definition of traversal is strictly stronger than \emph{every non-minimal element has a $<$-lesser neighbor}, but that they are equivalent for finite graphs, and more generally wellordered traversals.
   
   The fact that every finite connected graph admits a traversal is a consequence of the correctness of the following algorithm, called \emph{generic search} by Corneil and Krueger:
   
   Given as input a finite nonempty connected graph $(V,E)$, we initialize a sequence $S$ to be some single element of $V$. While there is a neighbor of $S$ not in $S$, we add it to the end of $S$. When there are no such neighbors, the algorithm halts. One can then prove that the basic traversal algorithm halts on all finite graphs, with $S$ a traversal of $(V,E)$. Indeed, we compute \emph{all} traversals of $(V,E)$ as we vary over the traces of this algorithm.
   
   This algorithm smoothly generalizes to infinite graphs. Suppose that $(V,E)$ is any nonempty connected graph, and initialize $S$ to be some single element of $V$. We build $S$ in ordinal-valued stages: at successor stages, add any vertex of $V \setminus S$ with a neighbor in $S$ to $S$ (if such a vertex exists), and at limit stages, take the union. One can prove that $S$ attains a fixed point at some ordinal stage, at which point it is a traversal of $V$. We simply call this algorithm \emph{nondeterministic graph search}, and as a consequence of its correctness we obtain the following
   
   \begin{thm} \label{existence}
   Every connected graph admits a traversal.
   \end{thm}
   
   (We will prove this statement when we discuss deterministic graph search below, but we note that it uses choice.)
   
   It is important to note that we do not recover all traversals as we vary over the traces of this nondeterministic algorithm, but only the \emph{wellordered} ones. In fact, it is a theme of this paper that the assumption of wellordering allows us to cleanly extend the theory of traversals from finite to infinite graphs.
   
   \subsection{Breadth- and depth-first traversals}
   There are many important refinements of the notion of traversal, probably the most fundamental of which are the \emph{breadth-first} and \emph{depth-first} traversals. Each has an associated search algorithm, namely \emph{breadth-first search} and \emph{depth-first search} respectively.
   Corneil and Krueger define breadth- and depth-first traversals for finite graphs as follows:
   
   \begin{defn}
   A traversal $<$ of a graph $(V,E)$ is \emph{breadth-first} in case for any three vertices $u < v < w$ such that $(u,w) \in E$ and $(u,v) \notin E$, there exists $x < u$ such that $(x,v) \in E$. 
   \end{defn}
   
   \begin{defn}
   A traversal $<$ of a graph $(V,E)$ is \emph{depth-first} in case for any three vertices $u < v < w$ such that $(u,w) \in E$ and $(u,v) \notin E$, there exists $x$ such that $u < x < v$ and $(x,v) \in E$. 
   \end{defn}
   
   In the case of a wellordering $<$, we can define the \emph{least neighbor function} as the function that takes any vertex (except the minimum) to its $<$-least neighbor. A wellordered traversal $<$ is breadth-first (according to the above definition) just in case its least neighbor function is weakly monotone, i.e., preserves $\le$. In Section \ref{BFT}, we will see that breadth-first search smoothly generalizes to infinite connected graphs, where it computes wellordered breadth-first traversals in the above sense. Therefore, we feel confident that the Corneil-Krueger definition of breadth-first traversal, at least for wellorderings, is the ``correct'' one.
   
   On the other hand, it is \emph{not} true that every connected graph admits a wellordered depth-first traversal, according the above definition: 
   \begin{prop}
     Let $T$ be the complete infinite binary tree of depth $\omega$, which we identify with $2^{<\omega}$. Make $T$ a graph by connecting each $u \in T$ to $u0$ and $u1$ by an edge. Then $T$ admits no depth-first traversal $<$.
   \end{prop}

	\begin{proof}
	It suffices to show there is no depth-first traversal starting at the root: if there were a depth-first traversal starting at vertex $v$, then consider its restriction to those elements above $v$ (assuming trees grow upward). We obtain a depth-first traversal of the subtree above $v$, which is isomorphic to $T$.
    
    As $T$ is acyclic, every vertex not the root must have a unique preceding neighbor, which in this case must be its parent.
   
   If there were a depth-first traversal $<$, the first $\omega$ elements would be some branch $(v_0, v_1, v_2, \dots)$ of the tree. Otherwise, consider the least $j \in \omega$ such that $v_{j-1}$ is not the parent of $v_j$. Any child $w$ of $v_{j-1}$ must appear after $v_j$ in the traversal ordering. Hence, $v_j$ is between $w$ and $v_{j-1}$, and these three vertices contradict the depth-first condition.
   
   Let $w_i$ be the child of $v_i$ that is not $v_{i+1}$. Then we claim that for each $i$, $w_{i+1}<w_i$, contradicting the assumption of wellordering. But if $w_i < w_{i+1}$ for some $i$, then $v_{i+1} < w_i < w_{i+1}$ contradicts the depth-first condition, since the only preceding neighbor of $w_i$ is $v_i$.
	\end{proof}

   The failure of every graph to admit a wellordered depth-first traversal seems to doom the prospect of any general inductive program reasonably implementing depth-first search over infinite structures, as the stages of any inductive definition are wellordered \cite{Mos74}. It does not resolve the question of whether every graph admits a depth-first traversal according to the Corneil-Krueger definition, nor what the ``correct'' definition is (there are several characterizations which are equivalent for finite graphs but inequivalent in general). In any case, we expect the analysis of infinitary depth-first search and traversals to be genuinely different from the finite case, as the assumption of wellordering which makes the theory extend so nicely for the other types of graph search fails here.
   
   \paragraph{Open questions} What is the correct definition of depth-first traversal for general, infinite graphs? Is there any sense in which depth-first search extends to infinite graphs? Alternatively, is there a different algorithmic way of producing depth-first traversals?
   
   \subsection{Deterministic searching}
    
    Nondeterministic graph search is nondeterministic because there is no canonical way to choose the next element of the traversal from the set of neighbors. However, if we consider the graph to be \emph{wellordered}, a natural choice of next element is to pick the \emph{least} neighbor at each stage. We define \emph{deterministic graph search} on an arbitrary wellordered connected graph $(V,E,<)$:
	\begin{defn}
		Define the set $S_{\alpha}$ for
		ordinal $\alpha$ by transfinite recursion:
		\begin{itemize}
			\item Let $S_0=\emptyset$ and $S_1 = \{v_{0}\}$ where $v_0$ is the $<$-least element of $V$.\footnote{Sometimes we will want to start the search on a given element $v$; we call this \emph{from} or \emph{starting with} $v$.}
			\item At successor stages, i.e., if $\alpha=\beta+1$, let $S_{\alpha}$
			be obtained from $S_{\beta}$ by adding the $<$-least element outside
			of $S_{\beta}$ but connected to $S_{\beta}$, i.e., the least element
			of
			\[
			\{v\in V\setminus S_{\beta}:\exists w\in S_{\beta}\ E(w,v)\}.
			\]
			If this set happens to be empty, let $S_{\alpha}=S_{\beta}$.
			\item At limit stages, let $S_{\alpha}=\bigcup_{\beta<\alpha}S_{\beta}$.
		\end{itemize}
		Let $\Gamma$ be the \emph{closure ordinal}; i.e., the least ordinal
		such that $S_{\Gamma}=S_{\Gamma+1}$. Notice that for $\alpha<\beta\le\Gamma$,
		$S_{\alpha}\subsetneq S_{\beta}$, but for $\Gamma\le\alpha<\beta$,
		$S_{\alpha}=S_{\beta}$.
	\end{defn}

The next few lemmas, show that ordering vertices according to the stage at which they are ``added to $S$" induces a linear ordering on $V$ that is both a traversal and a wellordering. We call the traversal of an ordered graphs defined by deterministic graph search the \emph{algorithmic traversal}.
	
	\begin{lem} $S_\Gamma = V$
	\end{lem}
	
	\begin{proof}
		Consider, by contradiction, the nonempty set $V \setminus S_\Gamma$. Either it contains some element with a neighbor in $S_\alpha$, in which case $S_\Gamma \subsetneq S_{\Gamma+1}$, or it has no such neighbor, making $(S_\alpha,V\setminus S_\alpha)$ a partition of $(V,E)$ into nonempty disconnected subsets.
	\end{proof}
	
	\begin{defn}
		For $v,w\in V$, let $v\prec w$ just in case the least $\alpha$ such that $v\in S_{\alpha}$
		is less than the least $\alpha$ such that $w\in S_{\alpha}$; i.e., 
		$$(\mu \alpha)(v \in S_\alpha)<(\mu \alpha)(w \in S_\alpha).$$
	\end{defn}
	
	\begin{lem} $\prec$ wellorders $V$. 
	\end{lem}
	
	\begin{proof}
		The map $v \mapsto (\mu \alpha)(v \in S_\alpha)$ is an order isomorphism between $(V,\prec)$ and the ordinal $\Gamma$.
	\end{proof}
	
	\begin{lem} $(V,\prec)$ is a traversal of $(V,E)$ 
	\end{lem}
	
	\begin{proof}
		By induction on $\alpha$, one can show that every $S_\alpha$ induces a connected subgraph of $(V,E)$.
	\end{proof}

	\paragraph{Motivating questions}
	The main purpose of this paper is to conduct a detailed analysis of the behavior of deterministic graph search and the properties of the algorithmic traversal. Towards this, we formulate two main questions which are the technical goals of our paper:
    
	\begin{itemize}
		\item We see that deterministic graph search is closed under wellorderings: the algorithmic traversal of a wellordered graph is again wellordered. One of the most important invariants of a wellordering is its order type. It is therefore natural to ask how deterministic graph search behaves with respect to the order type.
		\item Is there a sense in which, among all traversals of an ordered graph, the algorithmic traversal is canonical?
	\end{itemize}
	
    \paragraph{Contributions}
    We have made progress on both of our motivating questions:
    \begin{itemize}
    \item We prove a tight upper bound on the order type of an algorithmic traversal of a connected wellordered graph of some fixed order type. 
    \item We show that the algorithmic traversal is lexicographically least among the wellordered traversals of a wellordered connected graph.
    \item We are able to replicate both these results in the context of breadth-first traversals.
    \end{itemize}
    Furthermore,
    \begin{itemize}
    \item We find an apparently novel algorithm for producing the algorithmic traversal of a finite graph.
    \item We are able to define infinitary analogs of \emph{traversal trees}, which are so important in graph algorithms, and prove that they behave the same way as in the finite.
    \item We prove several other results which may be of independent interest, such as Theorems \ref{subsets-stable-under-traversals} and \ref{quotients-stable-under-traversals}, which give sufficient conditions for the algorithmic traversal ``factors through'' subgraphs and quotients, and Theorem \ref{preservation of cofinality} stating that deterministic search preserves cofinality in limit orders.
    \end{itemize}
    
    \section{Basic results and traversal trees} \label{basic-results}
    In this section we collect a number of basic results about deterministic graph search and the algorithmic traversal. Throughout this section, fix a wellordered connected graph $(V,E)$, and let $\prec$ be the algorithmic traversal obtained by deterministically searching $(V,E,<)$.
    
    \subsection{Functional properties of deterministic search}
    We briefly consider deterministic search as an operator on the space of wellordered traversals of a fixed graph.
    
	\begin{lem} \label{search fixes traversals}
    If $<$ is a traversal, then $\prec$ agrees with $<$.
	\end{lem}	
	\begin{proof}
		Let $v$ be the $<$-least element where $\exists w \prec v \ v < w$.  The set of elements at most $v$ in $(V,<)$ forms a connected subgraph, and the elements $<$-less then $v$ form a traversal initial segment by our choice of $v$, so any $S_\alpha$ containing $w$ also contains $v$.  This implies $v \prec w$, so no such $v$ exists.  Trichtomy strengthens $v<w \Longrightarrow v \prec w$ into $v<w \Longleftrightarrow v \prec w$.
	\end{proof}
	
	\begin{cor} The traversal algorithm is idempotent. As a consequence, it is a surjection from the set of wellorders to the set of wellordered traversals. \end{cor}
	\begin{proof} $(V,E,\prec)$ is a traversal, so is fixed by the traversal algorithm.
	\end{proof}
	
	\subsection{Existence of traversals}
    Here we refine Theorem \ref{existence} in Section \ref{intro}.
    
	\begin{lem} 
		\label{cardinal traversal order}
		If $(V,<)$ has cardinal order type $\kappa$, then $(V,\prec)$ has order type $\kappa$ as well.
	\end{lem}
	
	\begin{proof}
		%needs some clean-up
		If $\kappa$ is finite, we are done.  Otherwise, $\kappa$ is limit.  Suppose by contradiction that $W = V\setminus S_\kappa$ is nonempty.  Because $V$ is connected, there is an edge connecting some $w \in W$ and $v \in S_\delta$, for some $\delta < \kappa$. Then $\forall u \in S_\kappa \setminus S_\delta$, $u < w$.  Therefore, we have a set ($S_\kappa \setminus S_\delta$) which is not cofinal in $(V,<)$, but which has cardinality $\kappa$, a contradiction.
	\end{proof}
	
    As a consequence, we immediately see that
    
	\begin{cor} \label{exists traversal of least order type}
		Assuming choice, every connected graph has a traversal whose order type is the cardinality of its set of vertices.
	\end{cor}
	
	\subsection{Least neighbors and traversal trees}
	One of the most important invariants of a wellordered graph is its least neighbor function.
	
	\begin{defn}
		The \emph{least neighbor function} $p:V\to V$ is the function that takes a vertex $v$ to its $<$-least neighbor. It will be convenient to think of $p$ as undefined on the least element of $V$.
	\end{defn}
	
	As we remarked in Section \ref{intro}, having a lesser neighbor is equivalent to being a traversal under the assumption of wellordering:
	
	\begin{lem} \label{traversals are decreasing}
		$(V,<)$ is a traversal of $(V,E)$ just in case $\forall v \in V,\ p(v)<v$.
	\end{lem}
	
	\begin{proof}
		Suppose that for all non-minimal elements $v$, $p(v) < v$, but there is some initial segment $W$ of $(V,<)$ which induces a disconnected subgraph of $(V,E)$. We may assume that $W$ is the least such initial segment, and observe that $W$ cannot have limit order type. Let $w$ be the maximum element of $(W,<)$. Then since the subgraph induced by $W$ is disconnected but the subgraph induced by $W\setminus \{w\}$ is not, all neighbors of $w$ lie in $V\setminus W$. But since $W$ is initial, $p(w) > w$, a contradiction.
		
		Conversely suppose that $(V,<)$ is a traversal of $(V,E)$. Fix an arbitrary non-minimal $v \in V$, and consider the set ${w\in V: w \le v}$ of its predecessors. Since this set is initial, it induces a connected subgraph, so $v$ has some strictly lesser neighbor.
	\end{proof}
    
    If $(V,<)$ is a traversal, then $p$ also defines a spanning tree of the original graph:
	
   	\begin{lem} \label{traversal tree}
		Define the edge relation $E^\dagger$ to be the symmetrization of the graph of $p$. Then $(V,E^\dagger)$ is a spanning tree, i.e., an acyclic connected subgraph of $(V,E)$.
	\end{lem}
    
    	\begin{proof}
		Notice that for every $(v,w)\in E^\dagger$, there is a unique choice of $u_0$ and $u_1$ in $\{v,w\}$ such that $p(u_0)=u_1$. Call $(u_0,u_1)$ the \emph{orientation} of the edge $\{v,w\}$.
		
		Suppose there were a cycle in $(V,E^\dagger)$. Then it must either be a directed or undirected cycle, considering the orientation on the edges. But any undirected cycle must contain two arrows coming out of the same vertex, contradicting the functionality of $p$. Any directed cycle contradicts the property that $p(v) \prec v$ for any $v$.
		
		To show that $E^\star$ is spanning, it suffices to show that the minimal vertex $v_0$ is in the orbit of every element. But since on any other element $p(v) \prec v$, this follows.
	\end{proof}
    
    When $(V,<)$ is \emph{not} a traversal, we consider the least neighbor function of the associated algorithmic traversal.
    
	\begin{defn}
		The \emph{traversal least neighbor function} $p^\star:V\to V$ is the function that takes a vertex $v$ to its $\prec$-least neighbor, and is undefined on the least element of $V$. Notice that for any non-minimal vertex $v$, $p^\star (v) \prec v$, by Lemma \ref{traversals are decreasing}.
	\end{defn}
	
    The spanning tree associated with the traversal least neighbor function is the important graph-algorithmic concept of \emph{search tree}, which consists of the graph edges ``along which the search occurs,'' an idea made formal by Theorem \ref{search tree search}.
    
	\begin{thm} \label{search tree search}
		The traversal relation $\prec^\star$ obtained by traversing $(V,E^\star,<)$ is identical to $\prec$.
	\end{thm}
    
    \begin{proof}
    	Let $S^\star_\alpha$ represent the sequence of initial segments produced by the traversal algorithm on $(V,E^\star,<)$.  We show $S_\alpha = S^\star_\alpha$ by induction.  Suppose that for all $\mu < \alpha$, $S_\mu = S^\star_\mu$.  In the limit case $$S_\alpha = \bigcup\limits_{\mu < \alpha} S_\mu = S^\star_\alpha.$$  In the sucessor case, there is some $\beta$ where $\beta + 1 = \alpha$.  By assumption, $S_\beta = S^\star_\beta$.  We need to show that the least neighbor of $S_\beta$ in $(V,E,<)$ and $S^\star_\beta$ in $(V,E^\star,<)$ are the same. Since $E^\star \subseteq E$, it suffices to show that the $<$-least neighbor of $S_\beta$ in $E$ is already a neighbor in $E^\star$.  Let $v$ be that least neighbor. As $S_\beta$ is the set of all elements $\prec$-less than $v$ and $p^\star(v) \prec v$, $p^\star (v) \in S_\beta$ .  But $(p^\star (v),v) \in E^\star$ so v is a neighbor of $S^\star_\beta$ in $(V,E^\star,<)$, which concludes the proof.
    \end{proof}
	
	\subsection{Subgraphs and quotients}

	Here we give sufficient conditions under which the algorithmic traversals of certain subgraphs and quotients of a given graph are sub- or quotient orders of the algorithmic traversal of that graph.
    
    \begin{thm} \label{subsets-stable-under-traversals}
		Suppose $W \subseteq V$ is closed under $p^\star$, except for possibly at its $\prec$-least element $w_0$. Then $W$ induces a connected subgraph, and the algorithmic traversal $(W,\triangleleft)$ starting with $w_0$ of that subgraph agrees with $\prec$ on $W$.
	\end{thm}
    
    \begin{proof}
    The fact that $W$ induces a connected subgraph is a consequence of the fact that $w_0$ appears in the $p^\star$-orbit of every element of $W$ after finitely many steps.
    
    Suppose by contradiction that there are elements $w_1,\ w_2 \in W$ such that $w_1 \prec w_2$ but $w_2  \,\triangleleft\,  w_1$.
    We may assume that $w_1$ is the $\prec$-least element that occurs in such an inversion, and that $w_2$ is the $\triangleleft\,$-least among elements that occur in an inversion with $w_1$.
    
    It cannot be the case that $w_1 = w_0$, since $w_0$ is both $\prec$- and $\triangleleft\,$-least in $W$ and therefore cannot occur in any inversions. Similarly $w_2 \neq w_0$, since $w_0 \prec w_1 \prec w_2$.
    
    Let $w' = p^\star(w_1)$, so that $w' \in W$ and $w' \prec w_1 \prec w_2$. By minimality of $w_1$, $w' \triangleleft\, w_2$. Consider the stage where $w_2$ gets added in the algorithmic traversal of the subgraph induced by $W$. Since $w_1$ is connected to a previous element ($w'$), it must be the case that  $w_2 < w_1$.
    
    Now let $w''$ be the $\triangleleft\,$-least neighbor of $w_2$ in $W$, so that $w'' \triangleleft\, w_2 \triangleleft\, w_1$. By minimality of $w_2$, it must be the case that $w'' \prec w_1$. Consider the stage where $w_1$ gets added in the algorithmic traversal of the whole graph $(V,E,<)$. Since $w_2$ is connected to a previous element ($w''$), it must be the case that $w_1 < w_2$, a contradiction.
    
    \end{proof}

    \begin{thm} \label{quotients-stable-under-traversals}
    Suppose that $P$ is a partition of $(V,\prec)$ into intervals, and suppose that every member $W_i \in P$ induces a connected subgraph and is closed under $p^\star$, except for possibly at its least element $w_i$. Order $P$ by setting $W_i < W_j$ just in case $w_i < w_j$, and make it into a (connected) graph by putting an edge between $W_i$ and $W_j$ in case there is an edge between some element of $W_i$ and some element of $W_j$ in the original graph $(V,E)$.
    Let $(P,\triangleleft)$ be the algorithmic traversal order of this graph. Then $W_i \triangleleft\, W_j$ iff $w_i \prec w_j$.
    \end{thm}
    
    \begin{proof}
    Extend $\prec$ to $P$ by defining $W_i \prec W_j$ in case $w_i \prec w_j$. Suppose by contradiction that there is a pair $(i,j)$ such that $W_i \triangleleft\, W_j$ but $W_j \prec W_i$. Let $W_j$ be the $\prec$-least element of $P$ that occurs in such an inversion and $W_i$ be the $\triangleleft\,$-least element of $P$ that occurs in an inversion with $W_j$.
    
    Let $W_k$ be the $\prec$-least neighbor of $W_j$ in $P$. Then $W_k \prec W_i$, so by minimality of $W_j$, $W_k \triangleleft\, W_i$. Then at the stage where $W_i$ is added to the algorithmic traversal of $P$, $W_j$ is a candidate for addition, having a previous neighbor, $W_k$. Therefore, $W_j < W_i$, so $w_i < w_j$ in $V$.
    
    On the other hand, let $W_l$ be the $\triangleleft\,$-least neighbor of $W_i$. Then $W_l \triangleleft\, W_j$, so by minimality of $W_i$, $W_l \prec W_j$, so $w_l \prec w_j$ in $V$. Consider the stage where $w_j$ gets added to the algorithmic traversal of $V$. Since $W_i$ is connected to $W_l$, there must be an edge between $w_i$ and $W_l$. Otherwise, since $W_l$ is an interval, $p^\star(v)\notin W_i$ for some $v\neq w_i$ in $W_i$. Since $W_j$ is an interval, this neighbor of $w_i$ is $\prec$-less than $w_j$. Therefore, when $w_j$ gets added, $w_i$ is a candidate for addition, and $w_j < w_i$, a contradiction.
    \end{proof}

	\section{Algorithmic traversals and order type} \label{order-type}
	In this section we investigate the following question: given ordinal $\alpha$, what is an upper bound on the order type of algorithmic traversals of graphs with order type $\alpha$? The question of \emph{lower bounds} is also interesting, but we do not have much to say about it. 
	
	\begin{lem}
		\label{ordered subcofinality}
		Suppose set $V$ has orderings $<$ and $\prec$. For any set $C$ cofinal in $(V,<)$ there is a subset $C^\star$ of $C$ also cofinal in $(V,<)$ such that for any 
		$c_1, c_2 \in C^\star$ 
		$c_1 < c_2 \iff c_1 \prec c_2$.    
	\end{lem}

	\begin{proof}
		%done and reviewed
		Let $C$ be any cofinal subset of $(V,<)$.  We will define $C^\star$ as the range of some increasing function. Define a partial function $f:\mathcal{ON} \to C$ by recursion as follows: $$f(\alpha) = (\mu_\prec c \in C)(\forall \beta < \alpha\  f(\beta) < c).$$ If for some $\alpha$, the set $\{c \in C: \forall \beta < \alpha \  f(\beta) < c\}$ is empty, then $f$ is undefined for ordinal $\alpha$ and above.
		
		Define $C^\star$ to be the image of $f$. Clearly, $C^\star \subseteq C$. We claim that for $\gamma < \delta$ in the domain of $f$, both $f(\gamma) < f(\delta)$ and $f(\gamma) \prec f(\delta)$, which shows that $<$ and $\prec$ agree on $C^\star$.
		
		Fix $\gamma < \delta$ in the domain of $f$. By definition, $$f(\delta) \in \{c \in C: \forall \beta < \delta\ f(\beta)<c\},$$ all of whose elements bound $f(\gamma)$ from above. Therefore, $f(\gamma) < f(\delta)$. Moreover, 
		$$\{c \in C: \forall \beta < \delta\ f(\beta)<c\} \subseteq \{c \in C: \forall \beta < \gamma\ f(\beta)<c\}.$$ Taking the $\prec$-least element of both sides yields $f(\gamma) \preceq f(\delta)$, and $f(\gamma) \prec f(\delta)$ since they are unequal.
		
		It remains to show that $C^\star$ is cofinal in $(V,<)$, for which it suffices to show that it is cofinal in $(C,<)$. Otherwise there would exist some $t \in C$ bounding all elements of $C^\star$ from above.  
		%concise variant
		This would indicate that $f$ has domain $\ord$, as for any ordinal $\alpha$, $\{c \in C: \forall \beta < \alpha \  f(\beta) < c\}$ would contain $t$, so $f(\alpha)$ would be defined. However, $f$ is an injective function with a codomain that is a set, so it cannot have a proper class domain.
	\end{proof}
	
	\begin{thm}
		\label{preservation of cofinality}
		Given a well-ordered connected graph $(V,E,<)$ having limit order type and traversal ordering $(V,\prec)$, any $<$-cofinal subset of $V$ is also $\prec$-cofinal.
	\end{thm}
	
	\begin{proof}
		%done and reviewed
		Let $C$ be a $<$-cofinal subset of $V$.  By Lemma \ref{ordered subcofinality}, we can refine $C$ to a $<$-cofinal subset $C^\star$ on which $\prec$ agrees with $<$. It suffices to show that $C^\star$ is $\prec$-cofinal.
		
		Suppose $C^\star$ is not $\prec$-cofinal and consider the $\prec$-least element $t$ of $V$ $\prec$-above all elements of $C^\star$. Let $s=p^\star(t)$. Then $s \prec t$, so by minimality, there is some $c_1 \in C^\star$ such that $s \preceq c_1$.
		By $<$-cofinality of $C^\star$, there is some  $c_2 \in C^\star$ such that $t < c_2$.
		%Let $c_1$ and $c_2$ witness these statements respectively. 
		Pick some $c_3 \in C^\star$ which is $<$-greater than both of these, which is possible since $(C^\star,<)$ has limit order type. By transitivity, $t < c_3$. Additionally since $\prec$ and $<$ agree on $C^\star$, $s \prec c_3$.
		
		Let $\alpha$ be the stage of the traversal when $c_3$ is added. Then $s \in S_\alpha$, $t \notin S_\alpha$ as $c_3 \prec t$, and $c_3$ is the $<$-least neighbor of $S_\alpha$.
		
		However, $t$ is a neighbor of $S_\alpha$, as $t$ is connected to $s$. But $t < c_3$, which contradicts the minimality of $c_3$.
	\end{proof}
	
    We are now in a position to make our one statement about lower bounds: in the notation of the above theorem, the cardinality of $(V,\prec)$ is equal to the cardinality of $(V,<)$, and, if $(V,<)$ is limit, then the cofinality of $(V,\prec)$ is at least the cofinality of $(V,<)$.
    
	\begin{cor}
		\label{limit order preservation}
		The algorithmic traversal of any well-ordered connected graph with limit order type has limit order type.
	\end{cor}
	
	\begin{proof}
		%done
		Consider some well-ordered connected graph $(V,E,<)$ with limit order type and traversal ordering $(V, \prec)$.
		Suppose the traversal order type were not limit.  Then $V$ has $\prec$-maximal vertex $t$.  $V \setminus \{t\}$ is $<$-cofinal as $(V,<)$ has no maximal element, so has an element greater than $t$.  By Theorem \ref{preservation of cofinality}, $V \setminus \{t\}$ is $\prec$-cofinal, which is impossible because it does not contain $t$. 
	\end{proof}

	\begin{defn}
		Let $\zeta_\alpha$ denote the supremum of the algorithmic traversal order type of connected graphs with order type $\alpha$.
	\end{defn}
	
	\begin{defn}
		For any ordinal $\alpha$, $\zeta_\alpha$ is \emph{realized} if there is some well-ordered connected graph with order type $\alpha$ and traversal order type $\zeta_\alpha$.
	\end{defn}
	
	\begin{prop}
		\label{zeta increasing}
		$\zeta_\alpha$ is non-decreasing in $\alpha$.
	\end{prop}
	
	\begin{proof}
		It suffices to show that for any ordinals $\alpha<\beta$ and any connected graph $G$ on $\alpha$, there is a connected graph $H$ on $\beta$ whose algorithmic order type is at least that of $G$. We may define $H$ by making a copy of $G$ on the initial segment $\alpha \subseteq \beta$ and connecting all remaining vertices to the origin. Then the subgraph induced by the set $\alpha$ is connected and closed under the least neighbor function, so by Lemma \ref{subsets-stable-under-traversals}, the algorithmic traversal of $G$ agrees with the algorithmic traversal of $H$ on elements of the copy of $G$ inside $H$. In particular, the order type of the algorithmic traversal of $H$ is at least that of $G$.

	\end{proof}
	
	\begin{cor}
		\label{limit bound}
		For limit ordinals $\beta$, $\zeta_\beta = \sup_{\alpha<\beta} \zeta_\alpha$.
	\end{cor}

	\begin{proof}
		%done
		Because $\zeta$ is non-decreasing, $\sup_{\alpha<\beta} \zeta_\alpha \leq \zeta_\beta$.  We need to show $\zeta_\beta \leq \sup_{\alpha<\beta} \zeta_\alpha$.
		
		Consider any wellordered connected graph $(V,E,<)$ with input order type $\beta$.
		We want to show its algorithmic traversal order type $\Gamma$ satisfies $$\Gamma \leq \sup_{\mu < \beta} \zeta_\mu.$$
		Since $\Gamma$ is limit by Corollary \ref{limit order preservation}, 
		it suffices to show that $$(\forall \gamma < \Gamma)(\exists \mu < \beta) \  \gamma \leq \zeta_\mu.$$
		
		Towards which, fix arbitrary $\gamma < \Gamma$. Consider the set $S_\gamma \subseteq V$. Since it is not cofinal in the traversal ordering, by Theorem \ref{preservation of cofinality} it is also not cofinal in the input ordering. In particular, $(S_\gamma,<)$ has order type $\mu < \beta$.   
        However, $(S_\gamma,E,<)$ has traversal order type $\gamma$, %needs more justification%
        so $\gamma \le \zeta_\mu$, which concludes the proof.
	\end{proof}

	\begin{thm}
		\label{sum bound}
		For $\alpha, \gamma \in \ord$, $\zeta_{\alpha+\gamma} \leq \zeta_{\alpha}\cdot\zeta_{1+\gamma}$
	\end{thm}
	
	\begin{proof}
		%done
		It suffices to show that the order type of the algorithmic traversal of any connected graph $(V,E,<)$ of order type $\alpha + \gamma$ is at most $\zeta_{\alpha}\cdot\zeta_{1+\gamma}$.  Partition $V$ into $S\cup T$ such that $(S,<)$ has order type $\alpha$, $(T,<)$ has order type $\gamma$, and every element of $T$ is $<$-greater than every element of $S$.  Consider the ordered partition of $V$ induced by the set of all half-open intervals in $(V,\prec)$ between two consecutive elements of $T$ that contain the left endpoint in $T$, but not the right. (This partition also includes the set of all elements less than the least element of $T$.)
        Let $V_i$ range over the elements of this partition, and let $t_i \in T \cup \{v_0\}$ be the left endpoint of $V_i$.
		
		For any ordinals $j < i$, there are no edges between $V_i$ and $V_j$ except possibly those involving $t_i$. If there were an edge between some $v \in V_i \setminus \{t_i\}$ and some element of $V_j$, then $v$ must necessarily be added to the traversal before $t_i$, since $v < t_i$ in the input order and traversal initial segment $\bigcup_{k < i} V_k$ neighbors both.  Consequently, the subgraph induced by any partition $V_i$ must be connected, otherwise the subgraph induced by $\bigcup_{j\le i} V_j$ would not be connected, despite being a traversal-initial set. Moreover, $V_i$ is (trivially) closed under the traversal least-neighbor function except at its left endpoint $t_i$.
		
		Therefore, by Lemma \ref{subsets-stable-under-traversals}, the order $(V_i,\prec)$ is exactly the algorithmic traversal of the subgraph induced by $V_i$ starting at $t_i$. In particular, their order types are identical.
        
        Since each $V_i \setminus\{t_i\} \subseteq S \setminus \{v_0\}$, there is an order-preserving embedding $V_i \to S$ taking $t_i$ to be the least instead of greatest element of $V_i$. Therefore, the input order type when we traverse the subgraph induced by $V_i$ starting at $t_i$ is at most $\alpha$, and thus the traversal order type, and the order type of $(V_i,\prec)$ is at most $\zeta_\alpha$.
        
        We have bounded the order type of each interval $(V_i,\prec)$, but to bound the order type of $(V,\prec)$ it remains to determine the order type of the $V_i$'s. Slightly abusing notation, say that $V_i \prec V_j$ in case the interval $V_i$ precedes the interval $V_j$ in $(V,\prec)$.
        
        Consider the ``quotient graph'' $G$ obtained from $(V,E)$ as follows. The vertices of $G$ are identified with the sets $V_i$, and there is an edge between $V_i$ and $V_j$ in $G$ exactly when they are connected in $(V,E)$. Order the vertices of $G$ by defining $V_i < V_j$ when $t_i < t_j$. Then by Lemma \ref{quotients-stable-under-traversals} the relative order $\prec$ of the $V_i$'s is exactly the algorithmic traversal $\triangleleft$ of $(G,<)$.\
        
        Since the order type of $(G,<)$ is at most $1+\gamma$, the order type of $(G,\triangleleft)$ is at most $\zeta_{1+\gamma}$. Since the order type of each $(V_i,\prec)$ is at most $\zeta_\alpha$ and the relative order of the $V_i$'s is at most $\zeta_{1+\gamma}$, the order type of $(V,\prec)$ is at most $\zeta_\alpha \cdot \zeta_{1+\gamma}$, concluding the proof.
	\end{proof}

	In the following proofs we will need some standard facts about ordinals, which we collect below:
    \begin{itemize}
    \item Every ordinal $\alpha$ of cofinality $\kappa \ge \omega$ can be uniquely expressed as $\kappa \cdot \beta$ for some ordinal $\beta$.
    \item Every ordinal $\alpha$ can be uniquely expressed as $\omega \cdot \beta + n$ for some ordinal $\beta$ and $n < \omega$.
    \item Every ordinal $\alpha$ of cofinality $\kappa \ge \omega$ can be partitioned into $\kappa$ many subsets each of order type $\alpha$.
    \end{itemize}

    Furthermore, in what follows, we say that, e.g., $\zeta_\alpha$ is \emph{realized} in case there is a graph of order type (equivalently, on the ordinal) $\alpha$ the order type of whose algorithmic traversal is $\zeta_\alpha$.
    
	\begin{prop}
		\label{multiplication by finite part}
		For any infinite cardinal $\kappa$, ordinal $\sigma$ of cofinality $\kappa$, and ordinal $\nu$ with cardinality at most $\kappa$, %$|\nu| \leq \kappa$, 
		if $\zeta_{\sigma}$ and $\zeta_{1+\nu}$ are realized, then
		$$\zeta_{\sigma+\nu} = \zeta_{\sigma} \cdot \zeta_{1+\nu}$$
		and $\zeta_{\sigma+\nu}$ is realized.
	\end{prop}

	\begin{proof}
		By Theorem \ref{sum bound}, $$\zeta_{\sigma + \nu} \leq \zeta_{\sigma} \cdot \zeta_{1+\nu},$$ so it suffices to construct a graph on $\sigma + \nu$ the order type of whose algorithmic traversal is $\zeta_{\sigma} \cdot \zeta_{1+\nu}$.
		
		Partition $\sigma + \nu$ into $S \cup T$, where $S$ is the initial copy of $\sigma$ but without the element $0$, and $T$ is the final copy of $\nu$, with $0$. Then $S$ and $T$ have order types $\sigma$ and $1 + \nu$ respectively. By assumption there are connected ordered graphs $G_1$ and $G_2$ on the ordinals $\sigma$ and $1+\nu$ realizing $\zeta_\sigma$ and $\zeta_{1+\nu}$ respectively.
        
        Since $\sigma$ has cofinality $\kappa$, and $|1+\nu|\le\kappa$, we may partition $S$ into $|1+\nu|$ many subsets $S_i$ each of order type $\sigma$. Moreover there is a bijection associating each set $S_i$ with an element of $T$. For each $i$, connect the least element of $S_i$ to the corresponding element $t_i \in T$. The resulting graph $G$ is connected, and we claim that it has the desired property.
        
    	 Let $\prec$ be the algorithmic traversal of $G$ and $V_i = {t_i}\cup S_i$. Then we claim that each set $V_i$ is an interval in $\prec$. This is simply because every element of $S_i$ must occur after $t_i$, since all paths from $0$ to $S_i$ go through $t_i$, but every element of $S_i$ must precede $t_{i+1}$, the $\prec$-next element of $T$, because every element of $S_i$ is connected to $t_i$ via elements strictly smaller than $t_{i+1}$.
        
   	     Furthermore, each $V_i$ is closed under the $\prec$-least neighbor function except at $t_i$, for the simple reason that the \emph{only} neighbors of $S_i$ lie in $V_i$. Finally, each $V_i$ induces a connected subgraph of $G$, namely a copy of $G_1$ plus a single edge.
        
  	   Therefore, the hypotheses of Lemmas \ref{subsets-stable-under-traversals} \ref{quotients-stable-under-traversals} are satisfied. By Lemma \ref{subsets-stable-under-traversals}, the algorithmic traversal starting at $t_i$ of the ordered subgraph induced by $V_i$ is exactly $(V_i,\prec)$. But the order type of this traversal is exactly $1+\zeta_\sigma$.
        
  	   On the other hand, by Lemma \ref{quotients-stable-under-traversals}, the relative order among the $V_i$'s in the algorithmic traversal of $G$ can be determined by collapsing each to a single point, ordering them by their least element, connecting $V_i$ and $V_j$ whenever there was an edge between them in the original graph, and algorithmically traversing the resulting graph. But the this graph is exactly $G_2$, so the relative order among the $V_i$'s is exactly $\zeta_{1+\nu}$.
        
		Therefore, the algorithmic traversal order type of the entire graph $G$ is $(1+\zeta_\sigma)\cdot \zeta_{1+\nu}$. Notice that since $\sigma$ has limit order type by assumption, $\zeta_\sigma$, being the algorithmic traversal order type of $G_1$ is also limit by Corollary \ref{preservation of cofinality}. Hence $\zeta_\sigma = 1+\zeta_\sigma$, which explains the apparent discrepancy between what we stated and what we proved.
		
	\end{proof}

    Now we are in a position to fully resolve the question of upper bounds for algorithmic traversals.
	
	\begin{thm}
		\label{characterization of zeta}
		For any infinite ordinal $\alpha = \omega \beta + n$, $\zeta_\alpha = \omega^\beta \cdot (n+1)$ and $\zeta_\alpha$ is realized.
	\end{thm}

	\begin{proof}
		%done
		Any ordinal $\alpha$ may be written as $\omega \beta + n$ for a unique ordinal $\beta$ and finite ordinal $n$, so the term $\omega^\beta\cdot(n+1)$ is well-defined as a function of $\alpha$.
        We will prove that it is equal to $\zeta_\alpha$ by induction over infinite $\alpha$. The base case occurs when $\alpha = \omega$, and $\zeta_\omega = \omega$ is a consequence of Lemma \ref{cardinal traversal order}. ($\zeta_\alpha = \alpha$ for finite $\alpha$ is a consequence of the same.)
        
        Suppose that the theorem holds at all ordinals less than $\alpha = \omega\beta + n$. If $n > 0$ then by induction $\zeta_{\omega\beta}=\omega^\beta$ and it is realized. Moreover $\zeta_{n+1} = n+1$ and it is realized. The conclusion follows from Proposition \ref{multiplication by finite part}.
        
        On the other hand, if $n=0$, we split into cases depending on whether $\beta$ is a successor ordinal. If it is, and $\gamma + 1 = \beta$, then $\alpha = \omega\gamma + \omega$, and by induction $\zeta_{\omega\gamma} = \omega^\gamma$ and $\zeta_{1+\omega} = \omega$, and both are realized. The conclusion again follows from Proposition \ref{multiplication by finite part}.
        
        Finally, suppose that $n=0$ and $\beta$ is a limit ordinal. In this case $\alpha = \sup_{\gamma<\beta}{\gamma\omega}$, so by Corollary \ref{limit bound} and the inductive hypothesis, $$\zeta_\alpha = \sup_{\gamma<\beta}{\omega^\gamma}.$$ But the right hand side is exactly $\omega^\beta$.
        
        It remains to show that in this case, $\zeta_\alpha$ is realized. We will construct a graph that has induced subgraphs with order type $\zeta_\gamma$ for $\gamma$ cofinal in $\alpha$.
			
			Let the cardinal $\kappa$ be the cofinality of $\alpha$, and let $C \subseteq \alpha$ be cofinal with order type $\kappa$. Let $P$ be a partition of $\alpha$ of size $\kappa$, each element of which has order type $\alpha$, and fix a bijection between $\kappa$, $P$, and $C$. Let $\gamma$ range over elements of $C$, and $p_\gamma$ be the corresponding element of $P$.
            
            By induction, for each $\gamma$, there is a graph $G_\gamma$ on $\gamma$ with traversal order type $\zeta_\gamma$. Since each $p_\gamma$ has order type $\alpha$, there is an order-preserving injection from $\gamma$ to $p_\gamma$, and this allows us to copy of $G_\gamma$ on an initial segment of $p_\gamma$. (The resulting graph will not be connected, but this will not matter.)
            
            Finally, for any two elements of $P$, put an edge between their least elements. The resulting graph is not connected, but the algorithmic traversal of the connected component of 0 will have order type $\zeta_\alpha$, which suffices as $\lambda \mapsto \zeta_\lambda$ is monotone.
            
            To see this, note that the subgraph induced by any $p_\gamma$ is connected as a subgraph of the connected component of 0, and moreover $p_\gamma$ is closed under the least neighbor function except at its least element. Therefore by Lemma \ref{subsets-stable-under-traversals}, the algorithmic traversal of this subgraph, which has order type $\zeta_\gamma$, is the restriction of the algorithmic traversal of the whole graph to $p_\gamma$. Therefore, the algorithmic traversal of the whole graph has order type at least $\sup_{\gamma \in C}{\zeta_\gamma}$, but this is equal to $\zeta_\alpha$ by Corollary \ref{limit bound}.
            		
	\end{proof}

	\section{Algorithmic traversals are lexicographically extremal} \label{Lex-stuff}

	In this section, we ask if there is some canonical way to identify the algorithmic traversal within the space of all wellordered traversals of some wellordered graph. We find that there are two very natural answers to the first question, one that applies in general, and the other that applies only to finite graphs:
    
    \begin{itemize}
    \item The algorithmic traversal is \emph{lexicographically minimal} among all traversals of $(V,E,<)$.
    \item If $V$ is finite, the \emph{inverse} of the algorithmic traversal is \emph{colexicograhically maximal} among all traversals of $(V,E,<)$. (Colexicographic means ``reverse lexicographic.'')
    \end{itemize}
    
    Some explanation is in order. How do we compare two wellordered traversals lexicographically? While we usually regard a linear wellorder as a subset of $V^2$, we could of course also regard it as a bijection $t : \alpha \to V$ for some ordinal $\alpha$. Then given two such orders $t_1$ and $t_2$, we say that $t_1$ is lexicographically prior in case $t_1(\beta) < t_2(\beta)$, where $\beta$ is the least ordinal $\gamma$ satisfying $t_1(\gamma) \neq t_2(\gamma)$. (It cannot be the case that one is a strict subsequence of the other.)
    
    When $V$ is finite, the only possible ordinal $\alpha$ that occurs as the domain of a linear order is of course $n = |V|$. In this case, the set of algorithmic traversals of $V$ can be identified with a set of functions $n \to V$, or as a subset of $S_n$ if we identify $i \in n$ with the $i$-th element of $(V,<)$.
    
    \begin{prop} \label{lex-minimal}
    	The algorithmic traversal of a wellordered ordered connected graph is lexicographically least amongst all its traversals.
    \end{prop}
    
    \begin{proof}
    	Fix a wellordered graph $(V,E,<)$, and let $t:\Gamma \to V$ be its algorithmic traversal. Suppose that there is a lexicographically lesser traversal $t':\Gamma' \to V$, and suppose that $\gamma \in \Gamma \cap \Gamma'$ is the least ordinal on which they differ, so that $t'(\gamma) < t(\gamma)$. Let $S$ be the set of vertices $\{t(\delta):\delta < \gamma \}$. Then $S$ induces a connected subgraph and $t(\gamma)$ is the $<$-least neighbor of $S$. However, $t'(\gamma)$ is also neighbor of $S$ lesser than $t(\gamma)$, a contradiction.
    \end{proof}

    Suppose that $X \subseteq S_n$, and $X^{-1}$ is the set of inverses of elements of $X$. Suppose that $\mu^{-1}$ is the lexicographically least element of $X^{-1}$. Then what property does $\mu$ satisfy among all the elements of $X$?
    
    Let $X_0 \subseteq X$ be those set of $\sigma \in S_n$ which minimize $\sigma^{-1}(0)$. For $0\le i <n-1$, let $X_{i+1}$ be the subset of $X_i$ that minimize $\sigma^{-1}(i+1)$. Then $\mu \in X_n$; in fact, it is the unique member.
    
    Similarly, if $\mu^{-1}$ were colexicographically greatest, it is contained in $X_0$, where $X_{n-1} \subseteq X$ is the set of $\sigma \in S_n$ that maximize $\sigma^{-1}(n-1)$, and $X_{i-1}$ is the subset of $X_i$ that maximizes $\sigma^{-1}(i-1)$.
    
    \begin{thm} \label{colex-maximal}
    	The inverse of an algorithmic traversal of a finite ordered graph is colexicographically greatest among all inverses of traversals.
    \end{thm}
    
    \begin{proof}
    Let $n$ be the size of the graph $(V,E,<)$, and let $T \subseteq S_n$ be the set of its traversals. (Recall that if $R \subseteq V^2$ is a linear order, we identify it with the permutation that takes $i \in n$ to the index in $(V,<)$ of the $i$-th element of $R$.) Define $T_i$, $0 \le i \le n-1$ from $T$ as in the previous paragraph. We show that the algorithmic traversal $a$ is in each $T_i$, and indeed that it is the unique member of $T_0$.
    
    An \emph{pointed ordered partition} of $n = \{0,1,\dots,n-1\}$ is a linearly ordered partition of $n$ each of whose members has one designated element called the \emph{root}. Any pointed ordered partition induces a partial ordering on $n$, in which any two elements in different parts inherit the relative order of those parts, each root is least in its part, but other elements in the same part are not comparable. Given an ordered partition $\mathbf{P}$ and an element $\sigma \in S_n$, we say that $\sigma$ \emph{is consistent with} $\mathbf{P}$ in case $\sigma$ extends the partial order induced by $\mathbf{P}$. (In particular, each part of $\mathbf{P}$ is an interval in $\sigma$ whose left endpoint is the root.)
    
    For consistency in notation, let $T_n = T$. We will define a family of pointed ordered partitions $\mathbf{P}_i$ for $0\le i \le n$ such that $\mathbf{P}_{i-1}$ properly refines $\mathbf{P}_i$, $T_i$ is the set of exactly those traversals consistent with $\mathbf{P}_i$, and $a$ is consistent with each $\mathbf{P}_i$. This shows that $a$ is contained in $T_0$. The fact that it's the unique member comes from the observation that in a list of $n+1$ proper refinements of a partition of $n$, the last partition must be maximally refined (each part is a singleton), and thus is consistent with only one permutation.
    
    Let $\mathbf{P}_n$ be the singleton partition $\{n\}$ and let $0\in n$ be the root of its unique part. All traversals in $T$ are consistent with it, as they all start with 0.
    
    Suppose we know $\mathbf{P}_{i+1}$ for $0 \le i < n$ and we want to define $\mathbf{P}_i$. By induction assume, 
    \begin{itemize}
    \item $\mathbf{P}_{i+1}$ is a proper refinement of $\mathbf{P}_{i+2}$,
    \item each part of $\mathbf{P}_{i+1}$ induces a connected subgraph of $(V,E)$,
    \item deleting the root of each part disconnects the rest of the part from $0$,
    \item the set of roots is exactly $\{0,i+1,i+2,\dots,n-1\}$,
   	\item a traversal is in $T_{i+1}$ iff it's consistent with $\mathbf{P}_{i+1}$, and
    \item the algorithmic traversal $a$ is consistent with $\mathbf{P}_{i+1}$.
    \end{itemize}
    Let $X \in \mathbf{P}_{i+1}$ be the element containing $i$, and let $x$ be its root. By induction, $x \neq i$. Let $G$ be the subgraph of $(V,E)$ induced by $X$, which by induction is connected. Deleting $i$ disconnects $G$ into one or more connected components, one of which contains $x$. Let $X_0$ be that component which contains $x$ and $X_1$ be $X \setminus X_0$. Define $\mathbf{P}_i$ by replacing $X$ in $\mathbf{P}_{i+1}$ by $(X_0,X_1)$, in that order. Let $x$ and $i$ be the roots of $X_0$ and $X_1$ respectively. Then,
    \begin{itemize}
    \item $\mathbf{P}_i$ is a proper refinement of $\mathbf{P}_{i+1}$,
    \item both $X_0$ and $X_1$ induce connected subgraphs of $(V,E)$, and
    \item deleting $i$ disconnects the rest of $X_1$ from 0. (By induction, every path from $j \in X_1$ to $0$ contains $x$, and by construction every path from $j$ to $x$ contains $i$.) Furthermore,
    \item the set of roots is exactly $\{0,i,i+1,i+2,\dots,n-1\}$.
    \end{itemize}
    
    It remains to show that any traversal is in $T_i$ iff it is consistent with $\mathbf{P}_i$, and that $a$ is consistent with $\mathbf{P}_i$. 
    
    In any traversal, $i$ must precede any other element of $X_1$, as deleting $i$ disconnects $X_1$ from $0$. In the algorithmic traversal, all elements of $X_0$ precede $i$, since they follow $x$ in the traversal order (by induction), but are connected to $x$ by a path in $X_0$, all of whose elements are lesser than $i$. Hence the interval $X$ in $a$ factors into $(X_0,X_1)$, with $i$ being the left endpoint of $X_1$, and $a$ is consistent with $\mathbf{P}_i$.
    
    Traversals in $T_i$ are those in $T_{i+1}$ which maximize the index of $i$. Fix $t \in T_i$. By induction, $X$ induces an interval in $t$ with left endpoint $x$. By the previous paragraph, $i$ must precede any element of $X_1$ in $X$, but this maximum index is realized by $a$. Therefore, the interval $X$ in $t$ factors into $(X_0,X_1)$, and $t$ is consistent with $\mathbf{P}_i$.
    
    Conversely, suppose a traversal is consistent with $\mathbf{P}_{i}$. By induction, it is contained in $T_i$, and by assumption, the index of $i$ is as large as possible. Hence, it is contained in $T_{i+1}$, which concludes the induction.
    
    \end{proof}
        
    \paragraph{A different algorithm}    
    The proof of \ref{colex-maximal} suggests an alternative algorithm for computing the algorithmic traversal of a finite connected ordered graph $G$ starting with the vertex $v$:
    \begin{itemize}
    \item Let $w$ be the greatest element of $G \setminus \{v\}$, and let $X$ be the connected component of $v$ in the graph $G \setminus \{w\}$. Let $Y = G \setminus X$
    \item Recursively traverse the ordered graph induced by $X$ starting at $v$ and the ordered graph induced by $Y$ starting at $w$. Concatenate them to obtain the traversal of $G$ starting with $v$. 
    \end{itemize}
    
    This raises many questions. What is the complexity of this algorithm relative to standard graph representations? What happens if we modify the method of picking $w$ from $G\setminus \{v\}$? If we make it nondeterministic, do we recover all traversals as we vary over traces of the algorithm? If we let $w$ be the least element of $G \setminus \{v\}$, is the semantics well-behaved for infinite wellordered graphs? We leave these and other questions open.

	\section{Breadth-first search and traversals} \label{BFT}
    In this section, we replicate many of the results of Sections \ref{basic-results} through \ref{lex-minimal} in the context of breadth-first search and breadth-first traversals. Recall that
	
	\begin{defn}
		A wellordered traversal $\prec$ is \emph{breadth-first} if the least neighbor function is weakly monotone; i.e., for all vertices $v$, $v \preceq w \rightarrow p(v) \preceq p(w)$.
	\end{defn}
	
	We now define \emph{deterministic breadth-first search}, and prove that it computes a wellordered breadth-first traversal given a wellordered graph.
    
	\begin{defn}
		Define the set $B_\alpha$ and the sequence $Q_\alpha$ by transfinite recursion on $\alpha$. (We will variously regard $Q_\alpha$ as both a sequence and its underlying set.) Let $q_\alpha$ be the first element of $Q_\alpha \setminus B_\alpha$, given that it is nonempty.
		\begin{itemize}
			\item $B_0 = \emptyset$ and $Q_0 = (v_0)$.
			\item $B_{\alpha+1} = B_\alpha \cup \{q_\alpha\}$ and $Q_{\alpha+1}$ is obtained from $Q_\alpha$ by adding $\Gamma_{q_\alpha} \setminus Q_\alpha$ to the end in input order.
			\item For limit ordinals $\alpha$, $B_\alpha = \bigcup_{\beta < \alpha }B_\beta$ and $Q_\alpha  = \bigcup_{\beta < \alpha }Q_\beta$. For $\beta < \gamma$, $Q_{\gamma}$ is an end extension of $Q_\beta$, so the ordering on $Q_\alpha$ is inherited from its predecessors.
		\end{itemize}
		We see that for any ordinal $\alpha$, $B_\alpha$ is an initial segment of $Q_\alpha$ and $Q_\alpha$ is a wellorder. The closure ordinal $\Gamma$ is the least ordinal satisfying $B_\Gamma = B_{\Gamma+1}$.
	\end{defn}
	%note: the definition is simpler if $q_\alpha$ is element $\alpha$ of $Q_\alpha$, and this lets us define the entire queue ahead of time if we so choose, but as it is no longer the "queue" persay I am not sure if it is the "breath first search algorithm".
	%definition with expanding queue:
	%    \begin{defn}
	%		Define the set $B_\alpha$ and the sequence $Q_\alpha$ by transfinite recursion on $\alpha$. Let $q_\alpha$ be element $\alpha$ of $Q_\alpha$ if this exists.
	%		\begin{itemize}
	%			\item $B_0 = \emptyset$ and $Q_0 = (v_0)$.
	%			\item $B_{\alpha+1} = B_\alpha \cup \{q_\alpha\}$ if $q_\alpha$ exists or $B_\alpha$ otherwise.  $Q_{\alpha+1}$ is obtained from $Q_\alpha$ by adding $\Gamma_{q_\alpha} \setminus (B_\alpha \cup Q_\alpha )$ to the end in order.
	%			\item For limit ordinals $\alpha$, $B_\alpha = \bigcup_{\beta < \alpha }B_\beta$ and $Q_\alpha = \bigcup_{\beta\leq \gamma < \alpha} Q_\gamma$. For $\beta < \gamma$, $Q_{\gamma}$ is an end extension of $Q_\beta$, so the ordering on $Q_\alpha$ is inherited from its predecessors.
	%		\end{itemize}
	%		The closure ordinal $\Gamma$ is the least ordinal satisfying $B_\Gamma = B_{\Gamma+1}$.
	%	\end{defn}

	\begin{prop}
		$B_\Gamma = V$
	\end{prop}
	
	\begin{proof}
		Suppose by contradiction that $W = V \setminus B_\Gamma\neq\emptyset$. Since $(V,E)$ is connected, there is an edge between some $v \in B_\Gamma$ and $w \in W$. Let $\alpha < \Gamma$ satisfy $q_\alpha = v$. Then $w \in Q_{\alpha+1}$, so $w \in Q_\Gamma$. However, then $Q_\Gamma \setminus B_\Gamma$ is nonempty, so $B_{\Gamma+1} \neq B_\Gamma$, a contradiction.
	\end{proof}
	
	\begin{prop}
		The ordering $\prec$ of $V$ induced by $\alpha \mapsto q_\alpha$ is a breadth-first traversal.
	\end{prop}
	
	\begin{proof}
	Let $p^\circ$ be the least-neighbor function associated with $\prec$. It suffices to show that $p^\circ$ is monotone. By contradiction, suppose that $v \prec w$ but $p^\circ(w) \prec p^\circ(v)$. Let $q_\alpha = p^\circ(w)$. Neither $v$ nor $w$ is contained in $Q_\alpha$, otherwise one would have a neighbor in $B_\alpha$, which contains neither $p^\circ(v)$ nor $p^\circ(w)$.
		Similarly, $v$ is not contained in $Q_{\alpha+1}$, as it has no neighbors in $B_{\alpha+1} = Q_\alpha \cup p^\circ(w)$. On the other hand, $Q_{\alpha+1}$ contains $w$, being a neighbor of $q_\alpha$. Therefore, $Q_{\alpha+1}$ is an initial segment of $\prec$ containing $w$ but excluding $v$, contradicting $v \prec w$.
	\end{proof}

	We call $\prec$ the \textit{algorithmic breadth-first traversal}, and reserve $p^\circ$ for writing its associated least-neighbor function.
    
    The correctness of deterministic breadth-first search implies that every graph admits a breadth-first traversal, but it is not true that every graph admits a breadth-first traversal of order type the cardinality of the graph (Corollary \ref{BFT order type LB}). This is remarkable insofar as it is different from the case of regular traversals (Corollary \ref{exists traversal of least order type}).

	\paragraph{The breadth-first traversal tree}
    The function $p^\circ$ also admits an alternative characterization: $p^\circ(v) = q_\alpha$, where $v \in Q_{\alpha+1} \setminus Q_\alpha$. Said another way, whenever we add a new element $q_\alpha$ to $B_\alpha$, its set of neighbors $\Gamma_\alpha \setminus Q_\alpha$ that get added to the end of $Q_\alpha$ is exactly the inverse image of $q_\alpha$ under $p^\circ$.\footnote{Hence, the $p^\circ$-inverse image of singleton forms an interval in the algorithmic breadth-first traversal of order type at most the input order type.} This is easily proven by induction on $\alpha$.
    
    This characterization implies the breadth-first analogue to Lemma \ref{search tree search}. Suppose we define $E^\star \subseteq E$ to be the edge relation connecting each $v$ to $p^\circ(v)$. By Lemma \ref{traversal tree}, $(V,E^\star)$ is a spanning tree of $(V,E)$. Moreover, 
 
	\begin{lem} \label{BFT tree correctness}
		The algorithmic breadth-first traversals $\prec$ and $\prec^\star$ of $(V,E,<)$ and $(V,E^\star,<)$ respectively are identical.
	\end{lem}
	
	\begin{proof}
		Let $(B_\alpha)$ and $(Q_\alpha)$ be the stages of $\prec$ and $(C_\alpha)$ and $(R_\alpha)$ be the stages of $\prec^\star$. We will show by induction on $\alpha$ that $B_\alpha = C_\alpha$ and $Q_\alpha = R_\alpha$.
		It suffices to only consider successor stages of the induction. Suppose that $B_\alpha = C_\alpha$ and $Q_\alpha = R_\alpha$. Then $Q_\alpha \setminus B_\alpha$ and $R_\alpha \setminus C_\alpha$ must both be empty or both be nonempty. In the former case, each sequence has stabilized and we're done. In the latter case, the first elements of each are identical, so that $B_{\alpha+1} = C_{\alpha+1}$. If the common new element is $v$, then $Q_{\alpha+1}$ and $R_{\alpha+1}$ are obtained from the sequence $Q_\alpha = R_\alpha$ by adding the $E$-neighbors and $E^\star$ neighbors of $v$ respectively that are not already in $Q_\alpha$. Each $E^\star$-neighbor is already an $E$-neighbor, so it remains to show the converse. Each $E$-neighbor $w$ not in $Q_\alpha$ satisfies $p^\circ(w) = v$; hence, $w$ and $v$ are $E^\star$-neighbors and we're done.
		
	\end{proof}
    
    Lemma \ref{BFT tree correctness} is a useful tool in the analysis of the algorithmic breadth-first traversal for, as we shall see, breadth-first traversals of acyclic graphs are particularly well behaved.
	
	\subsection{Order and the algorithmic breadth-first traversal}

	It is a fundamental fact that breadth-first search ``goes level by level.'' In fact, if $(V,E)$ is acyclic, then \emph{any} breadth-first traversal can be partitioned into intervals of all elements a fixed distance from the least. We start out by proving this fundamental fact (indeed, something stronger), for arbitrary wellordered traversals of infinite acyclic graphs.

	\begin{prop} \label{BFT decomposition}
	Suppose that $v_0$ is a vertex in the acyclic graph $(V,E)$, $(V,\prec)$ is any wellordered breadth-first traversal starting at $v_0$, and $p$ is its least-neighbor function. Let $L_n$ be the set of nodes distance $n$ from $v_0$. Then for all $n<\omega$, $p$ maps $L_{n+1}$ into $L_n$,  $L_n$ is an interval in $(V,\prec)$, and (abusing notation) $L_n \prec L_{n+1}$.
	\end{prop}
    
    \begin{proof}
    First we show that $p$ maps $L_{n+1}$ into $L_n$. If not, let $w$ be the $\prec$-least element such that $w \in L_{n+1}$ but $p^\circ(w) \notin L_n$, for some $n$. Let $v = p(w)$. Then $v \in L_{n+2}$, since all neighbors of $w$ are in adjacent levels. In particular, $v \neq v_0$ and we can define $u = p(v)$, so that $u \prec v \prec w$. Since $w$ is the unique neighbor of $v$ in $L_{n+1}$, it must be the case that $u \in L_{n+3}$. But then $v \prec w$, $v \in L_{n+2}$, and $p(w) \in L_{n+3}$, contradicting the minimality of $w$.
    
    Next, we show that each $L_n$ is an interval in $(V,\prec)$. Assume by contradiction that $L_n$ is not an interval but $L_{n'}$ is for each $n' < n$. Then there is some configuration $u \prec v \prec w$, where $u \in L_n$, $v \in L_m$, and $w \in L_n$, for some $m \neq n$. Let $u'$, $v'$, and $w'$ be the images of $u$, $v$, and $w$ respectively under $p^\circ$, which exist because neither $n$ nor $m$ can be 0.
    
    Then $u' \preceq v' \preceq w'$, but since $v' \in L_{m-1}$ and $u', w' \in L_{n-1}$, $u' \prec v' \prec w'$, which contradicts the assumption that $L_{n-1}$ is an interval.
    
    Since $p : L_{n+1} \to L_n$ and each $L_n$ is an interval, it follows that every element of $L_n$ precedes every element of $L_{n+1}$ in $(V,\prec)$, concluding the proof.
    \end{proof}
    
    We now derive two corollaries of this result, one a lower bound for all breadth-first traversals of a certain type of graph, and the other an upper bound for the algorithmic breadth-first traversal of all graphs. Together they show that the upper bound is, in general, tight.
    
    \begin{cor} \label{BFT order type LB}
    Let $\kappa$ be a cardinal. Any breadth-first traversal of the complete $\kappa$-branching tree of depth $\omega$ (viewed as a graph), starting at the root, has order type at least $\kappa^\omega$.
    \end{cor}
    
    \begin{proof}
    Using the terminology of Proposition \ref{BFT decomposition}, it suffices to show that if $\prec$ is any breadth-first traversal, then the order type of $(L_n,\prec)$ is at least $\kappa^n$. We prove this by induction on $n$.
    This is true for $n=0$, where the order type of $L_n$ is $1 = \kappa^0$. For $n>0$, Proposition \ref{BFT decomposition} again says that $p$ is a weakly monotone map from $L_n$ to $L_{n-1}$. The latter has order type at least $\kappa^{n-1}$, by induction, and the inverse image of each element in $L_{n-1}$ has order type at least $\kappa$, by cardinality considerations. Therefore, the order type of $L_n$ is at least $\kappa^{n-1} \cdot \kappa = \kappa^n$.
    \end{proof}

	\begin{cor}
	Suppose $\prec$ is the algorithmic breadth-first traversal of $(V,E,<)$ and that $(V,<)$ has order type $\alpha$. Then $(V,\prec)$ has order type at most $\alpha^\omega$.
	\end{cor}

    \begin{proof}
    By Lemma \ref{BFT tree correctness}, it suffices to assume that $(V,E)$ is acyclic. By Propsition \ref{BFT decomposition}, $(V,\prec)$ admits an ordered partition into intervals $L_n$. It suffices to show that each $L_n$ has order type at most $\alpha^n$, which we do by induction on $n$.

	For $n=0$, the order type of $L_n$ is $1 = \alpha^0$. For $n>0$, Proposition \ref{BFT decomposition} again says that $p^\circ$ is a weakly monotone map from $L_n$ to $L_{n-1}$. The latter has order type at most $\alpha^{n-1}$ by induction, and the inverse image of each element in $L_{n-1}$ has order type at most $\alpha$. (This follows from the alternate characterization of $p^\circ$ above.) Therefore, the order type of $L_n$ is at most $\alpha^{n-1} \cdot \alpha = \alpha^n$.

    \end{proof}

%	On the other hand, this upper bound is tight for ordinals of regular cardinality.
%    \begin{prop}
%    	If ordinal $\alpha$ has regular cardinality, there is a graph on the ordinal $\alpha$ whose algorithmic breadth-first traversal has order type $\alpha^\omega$.
%    \end{prop}
%	\begin{prop} \label{BFT of trees}
%    	Any breadth-first traversal of a $\kappa$-branching tree of depth $\omega$, viewed as a graph, has order type at least $\kappa^\omega$.
%    \end{prop}
%	note: for cardinals, the maximum traversal order type is less than the maximum bredth first order type, as $\alpha < \alpha^\omega$.  %strict monotonicity of left multiplication implies strict monotocity of exponentiation with regards to the exponent
%	note: For $\alpha = \omega^\omega$, the maximum traversal order type is greater than the bredth first order type, as $\omega ^ {(\omega^\omega)}$ is greater than $(\omega^\omega)^\omega = \omega ^{(\omega^2)}$.

\subsection{Lexicographic minimality and functional properties}
Proposition \ref{lex-minimal} has an analogue in the breadth-first context. Recall that we identify wellorders of a set $V$ with bijections from some ordinal onto $V$. For any two distinct wellorders $f$ 
and $f'$ of $V$, there is a least ordinal $\xi$ such that $f(\xi)\neq f'(\xi)$. If $V$ comes with some linear order $<$, then there is a natural order on wellorders given by $f < f'\iff f(\xi) < f'(\xi)$, called the lexicographic order.

	\begin{prop} \label{BF lex-minimal}
		The algorithmic breadth-first traversal is lexicographically least among all breadth-first traversals.
	\end{prop}
	
	\begin{proof}

		Fix well-ordering $(V,E,<)$ with breadth first algorithmic traversal $t: \Gamma \to V$.  Suppose there is a lexicographically lesser breadth first traversal $t': \Gamma' \to V$.  Let $\gamma$ be the least element where $t$ and $t'$ disagree, so that $t'(\gamma) < t(\gamma)$ and $t'(\delta) = t(\delta)$ for $\delta < \gamma$. Let $W = t[\gamma]$, and $\alpha = \sup \{\beta:Q_\beta \subseteq W\}$. (Notice that $\alpha < \gamma$, so $q_\alpha = t(\alpha) = t'(\alpha)$.) Then $Q_\alpha \subseteq W$, since $$Q_\alpha = \bigcup_{\beta < \alpha} Q_\beta \subseteq W.$$

Let $p$ be the least neighbor function and $\sqsubset$ the ordering induced by $t'$. Let $v = t'(\gamma)$ and $w = t(\gamma)$. Then $v,w \notin Q_\alpha$, since $v,w \notin W$, and $v<w$ by assumption.  All neighbors of elements $\sqsubset$-less than $q_\alpha$ are in $Q_\alpha$, so both $p(v),p(w) \sqsupseteq q_\alpha$.

Since $w$ is a neighbor of $q_\alpha$, $p(w) = q_\alpha$. Since $p$ is weakly monotone and $v \sqsubset w$, $p(v) = q_\alpha$, so $v$ is a neighbor of $q_\alpha$ as well.

However, $w$ is the $<$-least neighbor of $q_\alpha$ that is not contained in $Q_\alpha$, contradicting $v<w$.
	\end{proof}

From this we can easily deduce the analogue of Lemma \ref{search fixes traversals}:

	\begin{cor}
	If $(V,E,<)$ is a wellordered graph with algorithmic breadth-first traversal $\prec$, and $<$ is 		already a breath-first traversal, then $<$ and $\prec$ agree.
	\end{cor}

Notice that among the wellorderings of $V$ induced by $<$, the wellordering $<$ itself is lexicographically minimal. Since it happens to be a breadth-first traversal by assumption, Proposition \ref{BF lex-minimal} guarantees that it is the algorithmic breadth-first traversal of $(V,E,<)$.

Hence, if we view breadth-first search as computing a functional on the space of wellorderings of a fixed graph, then this functional surjects onto the set of breadth-first traversals, and is idempotent.

Finally, we raise the question of whether there are any functional relationships between graph search and breadth-first search. Fix a graph $(V,E)$, and let $\mathbf{B} \subseteq \mathbf{T} \subseteq \mathbf{W}$ be its set of breadth-first traversals, traversals, and wellorderings respectively. Let $\tau$ and $\beta$ be the functionals operating on $\mathbf{W}$ defined by deterministic graph search and breadth-first search respectively. Then $$\beta = \tau \circ \beta,$$ trivially, since $\beta$ maps into a set fixed by $\tau$.

One might hope that $\beta = \beta \circ \tau$ as well; as a breadth-first traversal is a refinement of the notion of traversal, it would be nice if the equivalence $\beta(f) = \beta(f')$ were a refinement of $\tau(f) = \tau(f')$. Unfortunately, this is false. Consider the graph on $6$ elements obtained by connecting a 5-cycle to a single vertex by one edge. Order the graph by labeling the 5-cycle $(0,1,2,4,5)$, labeling the single vertex $3$, and connecting it to $5$. Then if we apply $\beta$ to this ordering, we get $(0,1,5,2,3,4)$, whereas if we apply $\beta \circ \tau$, we get $(0,1,5,2,4,3)$.

At the same time it seems like the identity $\beta = \beta \circ \tau$ is often true, in that constructing counterexamples is tricky. We wonder if there is any meaningful theorem in this direction, or whether any other nontrivial functional identities hold of $\beta$ and $\tau$.

\end{document}